\documentclass[12pt]{article}
\usepackage[
margin=2in,
includefoot,
footskip=12pt,
]{geometry}
\usepackage{layout}
\usepackage{amsmath}
\usepackage{amsfonts}
\usepackage{amssymb}
\usepackage{amsthm,fullpage}
\usepackage{graphics}
\usepackage{float,graphicx}
\usepackage{tikz}
\usepackage{hyperref}

\theoremstyle{plain}
\newtheorem{thm}{Theorem}[section]
\newtheorem{lem}[thm]{Lemma}

\theoremstyle{definition}

\newtheorem*{ex}{Example}
\newtheorem{prop}[thm]{Proposition}

\title{Constructing cospectral graphs by unfolding non-bipartite graphs\footnote{Some of the results of this article are part of the M.S. thesis of last author \cite{hitesh}.}}
\author{
M. Rajesh Kannan \thanks{\texttt{rajeshkannan@math.iith.ac.in, rajeshkannan1.m@gmail.com}, Department of Mathematics, Indian Institute of Technology Hyderabad, Hyderabad 502284, India} \and 
Shivaramakrishna Pragada 
\thanks{\texttt{shivaramkratos@gmail.com}, Department of Mathematics, Simon Fraser University, Burnaby, BC, V5A 1S6, Canada } 
\and 
Hitesh Wankhede 
\thanks{\texttt{hiteshwankhede9@gmail.com}, Department of Mathematics, Indian Institute of Science Education and Research Pune, Pune 411 008, India} 
\thanks{ Theoretical Computer Science Group, The Institute of Mathematical Sciences (HBNI), Chennai 600 113, India}}
\date{\today}
\begin{document}
    \maketitle
    \baselineskip=0.25in

\begin{abstract}
	In 2010, Butler \cite{but-lama} introduced the unfolding operation on a bipartite graph to produce two bipartite graphs, which are cospectral for the adjacency and the normalized Laplacian matrices. In this article, we describe how the idea of unfolding a bipartite graph with respect to another bipartite graph can be extended to nonbipartite graphs. In particular, we describe how unfoldings involving reflexive bipartite, semi-reflexive bipartite, and multipartite graphs are used to obtain cospectral nonisomorphic graphs for the adjacency matrix. 
\end{abstract}

{\bf AMS Subject Classification(2010):} 05C50.

\textbf{Keywords.} Adjacency matrix, Cospectral graphs,   Partitioned tensor product, Reflexive and semi reflexive graphs, Unfolding.
\section{Introduction}

We consider simple and undirected graphs. Let $G=(V(G),E(G))$ be a graph with the vertex set $V(G)=\{1,2,\ldots, n\}$ and the edge set $E(G)$. If two vertices $i$ and $j$ of $G$ are adjacent, we denote it by $i\sim j$. For a vertex $v \in V(G)$, let $d_G(v)$ denote the degree of the vertex $v$ in $G$. For a graph $G$ on $n$ vertices, the \textit{adjacency matrix} $A(G)=[a_{ij}]$ is an $n \times n$ matrix defined by
$$a_{ij}=\begin{cases}1,& \mbox{if}~\ i\sim j,\\0,& \mbox{otherwise}.\end{cases}$$ The \textit{spectrum }of a graph $G$ is the set of all eigenvalues of $A(G)$, with corresponding multiplicities. Two graphs are \textit{cospectral} if the corresponding adjacency matrices have the same spectrum. If any graph which is cospectral with $G$ is also isomorphic to it, then $G$ is said to be \textit{determined by its spectrum} (DS graph for short); otherwise, we say that the graph $G$ has a cospectral mate or we say that $G$ is not determined by its spectrum (NDS for short). To show that a graph is NDS, we construct a cospectral mate. It has been a longstanding problem to characterize graphs that are determined by their spectrum. In \cite{haemers_are_2016}, Haemers conjectured that almost all graphs are DS. Recently, in \cite{arvind2023hierarchy}, Arvind et al. proved that almost all graphs are determined up to isomorphism by their eigenvalues and angles. For surveys on DS and cospectral graphs, we refer to \cite{vandam-haemers, devel-vandam-haemers}. 


The other face of the conjecture about DS graphs is the problem of constructions of cospectral graphs. A well-known method to construct cospectral graphs is Godsil-Mckay switching \cite{godsil-mckay-1982}. Recently, an analog of the switching method was introduced by Wang, Qiu, and Hu \cite{wang2019cospectral, qiu2020theorem}. In \cite{but-lama}, Butler introduced a cospectral construction method based on unfolding a bipartite graph, which works for both normalized Laplacian and adjacency matrices. Later in \cite{kannan-pragada}, Kannan and Pragada generalized this construction and extended the idea to obtain three new cospectral constructions. In \cite{ji-gong-wang}, Ji, Gong, and Wang have generalized the unfolding idea further and given a characterization of isomorphism for their construction. In \cite{godsil-mckay-1976}, Godsil and McKay constructed cospectral graphs for the adjacency matrices using the partitioned tensor product of matrices. Unifying ideas of the articles \cite{godsil-mckay-1976}, \cite{hammack} and \cite{ji-gong-wang}, in  \cite{kph-unf}, we  characterized the isomorphism case of very general unfolding operation on bipartite graphs. In this paper, we use partitioned tensor products to describe unfolding constructions involving reflexive bipartite, semi-reflexive, and multipartite graphs. We also address the isomorphism of the constructed cospectral graphs.


The outline of this paper is as follows: In Section \ref{prelims}, we include some of the needed known results for graphs and matrices. Section \ref{Construct_God_Mckay} discusses the isomorphism case of the construction in \cite{godsil-mckay-1976} involving reflexive bipartite graphs. In Section \ref{Semi_reflex_unfold}, we discuss unfolding involving semi-reflexive bipartite graphs, which extends the Construction III of \cite{kannan-pragada}. Section \ref{Multipart_unfold} is devoted to the study of multipartite unfolding in which we use the partial transpose operation discussed in \cite{dutta}.

\section{Preliminaries}\label{prelims}

The notion of partitioned tensor product of matrices is used extensively in this article. This is closely related to the well-known Kronecker product of matrices. The \emph{Kronecker product} of matrices $A = (a_{ij})$ of size $m \times n$ and $B$ of size $p \times q$, denoted by $A \otimes B$, is the $mp \times nq$ block matrix $(a_{ij}B)$.

The\textit{ partitioned tensor product }of two partitioned matrices $M = \begin{bmatrix}
	U & V \\
	W & X \\
\end{bmatrix}$ and $H = \begin{bmatrix}
	A & B \\
	C & D \\
\end{bmatrix}$, denoted by $M  \underline{\otimes} H$, is defined as $ \begin{bmatrix}
	U\otimes A & V\otimes B \\
	W\otimes C & X\otimes D \\
\end{bmatrix}.$ For the matrix $M\underline{\otimes} H$, the partition defined above is called the canonical partition of the matrix $M \underline{\otimes} H$.   

Given the matrices $U$, $V$, $W$ and $X$, define $\mathcal{I}(U,X)= \begin{bmatrix}U& 0\\0 &X\end{bmatrix}$ and $\mathcal{P}(V,W)= \begin{bmatrix}0 &V\\W& 0\end{bmatrix}$ where $0$ is the zero matrix of appropriate order. A $2\times2$ block matrix is called a \textit{diagonal} (resp., \textit{an anti diagonal}) block matrix if it is of the form $\mathcal{I}(U, X)$ (resp., $\mathcal{P}(V, W)$).   The above notions were introduced by Godsil and McKay \cite{godsil-mckay-1976}. The following proposition is easy to verify.

\begin{prop}\label{parti-diag}
	Let $Q$ and $R$ be the matrices of the form $\mathcal{I}(Q_1,Q_2)$ and $\mathcal{I}(R_1,R_2)$, respectively. If $M = \begin{bmatrix}
		U & V \\
		W & X \\
	\end{bmatrix}$ and $H = \begin{bmatrix}
		A & B \\
		C & D \\
	\end{bmatrix}$ are $2 \times 2$ block matrices, then $$(Q \underline{\otimes} R)(M \underline{\otimes} H)=(QM) \underline{\otimes} (RH).$$ The same holds true when the matrices $Q$ and $R$ are both of the form $\mathcal{P}(Q_1,Q_2)$ and $\mathcal{P}(R_1,R_2)$, respectively.
\end{prop}

Two matrices $A$ and $B$ are said to be \textit{equivalent}, if there exist invertible matrices $P$ and $Q$ such that $Q^{-1}AP=B$. If the matrices $P$ and $Q$ are orthogonal, then matrices $A$ and $B$ are said to be \textit{orthogonally equivalent}. If the matrices $P$ and $Q$ are permutation matrices, then matrices $A$ and $B$ are said to be \textit{permutationally equivalent}. Using the singular value decomposition, it is easy to see that any square matrix is orthogonally equivalent to its transpose.

A square matrix $A$ is said to be a \textit{PET} (resp. \textit{PST}) matrix if it is {permutationally equivalent} (resp. {similar}) to its transpose. If the set of row sums of an $n \times n$ matrix $A$ is different from the set of column sums of $A$, then $A$ is non-PET.

Next, we recall the cancellation law of matrices given by Hammack.

\begin{thm}[{\cite[Lemma 3]{hammack}}]\label{cancellation} Let $A$, $B$ and $C$ be $(0,1)$-matrices. Let $C$ be a non-zero matrix and $A$ be a square matrix with no zero rows \footnote{The assumption that \lq $A$ has no zero rows\rq can also be replaced with the assumption \lq $A$ has no zero columns\rq (see \cite{hammack}). Hence, we interpret this assumption as \lq $A$ cannot have both a zero row and a zero column\rq.}. Then, the matrices $C\otimes A$ and $C\otimes B$ are permutationally equivalent if and only if $A$ and $B$ are permutationally equivalent. Similarly, the matrices $A\otimes C$ and $B\otimes C$ are permutationally equivalent if and only if $A$ and $B$ are permutationally equivalent.
\end{thm}

Let $G$ be a bipartite graph with vertex partition $V(G)=X\cup Y$; $G$ is  \textit{semi reflexive} if each vertex in either $X$ or $Y$ has a loop, and \textit{reflexive} if each vertex in $V(G)$ has a loop. If the degrees of all vertices in one of the partite sets is $k$ and the degrees in the other is $l$, then $G$ is said to be $(k,l)$-biregular. 
Two graphs $G_1$ and $G_2$ are isomorphic if and only if the corresponding adjacency matrices $A(G_1)$ and $A(G_2)$ are permutationally similar. An \textit{automorphism} of a graph $G$ is an isomorphism from the graph $G$ to itself.  Every automorphism of a graph $G$ on $n$ vertices can be represented by an $n \times n$ permutation matrix $P$ such that $P^TA(G)P=A(G)$.

\section{Construction I - Unfoldings involving a reflexive bipartite graph}\label{Construct_God_Mckay}

We first recall a cospectral construction by Godsil and Mckay \cite{godsil-mckay-1976}. The main objective of this section is to investigate conditions under which these constructed cospectral graphs are isomorphic. Let $V$ and $W$ denote matrices of size $m\times n$ and $n\times m$ respectively. Let $I_m$ and $I_n$ denote identity matrices of the order $m$ and $n$, respectively. Also, let $A$, $B$, $C$ and $D$ be matrices of size $p\times p$, $p\times q$, $q\times p$ and $q\times q$ respectively. Define the partitioned matrices  $L=\begin{bmatrix}I_m&V\\W&I_n\end{bmatrix}$, $H=\begin{bmatrix}A&B\\C&D\end{bmatrix}$ $H^{\#}=\begin{bmatrix}D&C\\B&A\end{bmatrix}$, and the partitioned tensor products \begin{equation}\label{equ-sec1.1} L\underline{\otimes} H=\begin{bmatrix}I_m\otimes A&V\otimes  B\\W\otimes C&I_n\otimes D\end{bmatrix}, ~~~~~~\;L\underline{\otimes}  H^{\#}=\begin{bmatrix}I_m\otimes D&V\otimes C\\W\otimes B&I_n\otimes  A\end{bmatrix}. \end{equation}
For the partitioned tensor products, the following result is known.
\begin{thm}[{\cite{godsil-mckay-1976}}]\label{gm-cospec}
	The matrices $L\underline{\otimes} H$ and $L\underline{\otimes} H^{\#}$ have the same eigenvalues if and only if either $m=n$ or the blocks $A$ and $D$ have the same eigenvalues. 
\end{thm}

For the proof of this theorem, we refer to \cite{godsil-mckay-1976}. In this section, we consider a special case of this construction. We assume that the partitioned matrices $L$ and $H$ are adjacency matrices of some simple graphs, that is, $W=V^T$, $C=B^T$ and $A$ and $D$ are symmetric and have zero diagonal entries. {For an $n \times n$ symmetric $(0, 1)$-matrix $A$ with zero diagonal entries, let $G_A$ denote the simple graph whose adjacency matrix is $A$.} Since the graphs $G_A$ and $G_D$ do not have any loops, $G_{L\underline{\otimes} H}$ and $G_{L\underline{\otimes} H^{\#}}$ also do not have any loops. Under the assumptions made, the involved partitioned tensor products are given by:$$L\underline{\otimes} H=\begin{bmatrix}I_m\otimes A&V\otimes  B\\V^T\otimes B^T&I_n\otimes D\end{bmatrix}, \;L\underline{\otimes}  H^{\#}=\begin{bmatrix}I_m\otimes D&V\otimes B^T\\V^T\otimes B&I_n\otimes  A\end{bmatrix}.$$ 

We call $G_{L\underline{\otimes} H}$ and $G_{L\underline{\otimes} H^{\#}}$ the graphs obtained by unfolding the graph $G_H$ with respect to reflexive bipartite graph $G_L$. The vertex partition induced naturally by the partitioned tensor product on the graph $G_{L\underline{\otimes} H}$ is defined to be the canonical vertex partition of $G_{L\underline{\otimes} H}$. In the next lemma, we assume that $V$ is a square matrix and give a short proof for the cospectrality of $G_{L\underline{\otimes} H}$ and $G_{L\underline{\otimes} H^{\#}}$.

\begin{lem}\label{thm-unf-ref-cospec}
	If $V$ and $W$ are square matrices, then the graphs $G_{L\underline{\otimes}H}$ and $G_{L\underline{\otimes}H^{\#}}$ are cospectral for the adjacency matrix.  
\end{lem}
\begin{proof}
	Since $V$ is a square matrix, there exist two orthogonal matrices $Q_1$ and $Q_2$ such that $Q_1^TVQ_2=V^T$. Define $Q=\mathcal{P}(Q_1,Q_2)$. Then $Q$ is orthogonal and $Q^TLQ=L$. 
	Let $R=\mathcal{P}(I_p,I_q)$, then $R$ satisfies $R^THR=H^{\#}$. Define $S=Q\underline{\otimes} R$, then $S$ is an orthogonal matrix. Now using Proposition \ref{parti-diag}, we get $S^{T}(L\underline{\otimes} H)S=L\underline{\otimes} H^{\#}$. Thus $G_{L\underline{\otimes}H}$ and $G_{L\underline{\otimes}H^{\#}}$ are cospectral. 
\end{proof}

Next, we investigate conditions under which the constructed cospectral graphs given by Theorem \ref{gm-cospec} are isomorphic. Given two graphs $G$ and $H$ with the vertex partitions $V(G) = X \cup Y$ and $V(H) = V \cup W$, we say an isomorphism $f$ from $G$ to $H$ \emph{respects the partition} if $f(X) = V$ or $f(X) = W$. In the next two results, the matrices $A, B, C, D$ and $V$ are defined as in equation (\ref{equ-sec1.1}).


 


\begin{lem}\label{prop:eta}
	Let $G_{L}$ be a reflexive $(k,l)$-biregular bipartite graph with $k \neq l$. If the graphs $G_{L\underline{\otimes}H}$ and $G_{L\underline{\otimes}H^{\#}}$ are isomorphic, then any isomorphism between them respects the canonical vertex partitions of $G_{L\underline{\otimes}H}$ and $G_{L\underline{\otimes}H^{\#}}$.
\end{lem}

\begin{proof}
	Let the graphs $\Gamma_1=G_{L\underline{\otimes} H}$ and $\Gamma_2=G_{L\underline{\otimes} H^{\#}}$ be isomorphic, and let $f$ be an isomorphism from $\Gamma_1$ to $\Gamma_2$. Let $V(\Gamma_1)=X_1\cup Y_1$ and $V(\Gamma_2)=X_2\cup Y_2$ be the canonical vertex partitioning of the graphs $\Gamma_i$ for $i=1,2$. Let $b_i$ and $b_i'$ denote the $i^{th}$ row sum of the matrices $B$ and $B^T$, respectively. Let $a_i$ and $d_i$ denote the $i^{th}$ row sums of $A$ and $D$ respectively. Let $V$ have constant row sum $l$ and column sum $k$. 
	
	Suppose $k<l$.  Let $x\in X_1$ be the vertex of maximum degree in this set. Then, we will show that $f(x)\in X_2$. On the contrary, suppose $f(x)\in Y_2$. Then, $d_{\Gamma_1}(x)=a_i+lb_i$ for some $1\leq i\leq p$ and $d_{\Gamma_2}(f(x))=kb_j+a_j$ for some $1\leq j\leq p$. Since the isomorphism preserves the degrees, we have $a_i+lb_i=kb_j+a_j$. But as $k<l$, we have $a_i+lb_i < a_j +lb_j$; this is a contradiction as $x$ has a maximum degree in $X_1$. Thus $f(x)\in X_2$. 
	
	Let $x_1,\ldots, x_{rm}$ be the vertices in $X_1$ with the same maximum degree such that \\ $d_{\Gamma_1}(x_{1+(s-1)m})=\ldots=d_{\Gamma_1}(x_{m+(s-1)m})=a_{i_s}+lb_{i_s}$ for $s\in \{1,2,\ldots, r\}$ where $a_{i_1}+lb_{i_1}=\ldots=a_{i_r}+lb_{i_r}$ for $1\leq i_1,\ldots, i_r\leq p$. Then, using the previous argument $f(x_1),\ldots, f(x_{rm})$ are vertices in $X_2$ such that $d_{\Gamma_2}(f(x_{1+(s-1)m}))=\ldots=d_{\Gamma_2}(f(x_{m+(s-1)m}))=d_{j_s}+lb'_{j_s}$ for $s\in \{1,2,\ldots, r\}$ where $d_{j_1}+lb'_{j_1}=\ldots=d_{j_r}+lb'_{j_r}$ for $1\leq j_1,\ldots, j_r\leq p$. Let $B'$ and $B''$ be the matrices obtained by removing $i_s^{th}$ row and $j_s^{th}$ column respectively from $B$ for all $s\in \{1,2,\ldots, r\}$. Define $A'$ and $D''$ to be the matrices obtained by removing $i_s^{th}$ row and $j_s^{th}$ row from $A$ and $D$ respectively for all $s\in \{1,2,\ldots, r\}$. Define $\Gamma_1'$ and $\Gamma_2'$ to be the induced graphs corresponding to the adjacency matrices $$\begin{bmatrix}I_m\otimes A'&V\otimes  B'\\V^T\otimes  B'^T&I_n\otimes  D\end{bmatrix}\; \text{and} \begin{bmatrix}I_m\otimes D''&V\otimes  B''^T\\V^T\otimes  B''&I_n\otimes  A\end{bmatrix}$$ respectively. Note that the sizes of the matrices $B'$ and $B''$ are $(p-r)\times q$ and $q\times (p-r)$ respectively. Now $\Gamma_1'$ and $\Gamma_2'$ are isomorphic as well, apply the same argument for $\Gamma_1'$ and $\Gamma_2'$ until the graphs reduce to $G_{I_n\otimes D}$ and $G_{I_n\otimes A}$ respectively. Thus $f(X_1)=X_2$ and hence $f(Y_1)=Y_2$. The proof for the case $k>l$ is similar. 
\end{proof}

In the following theorem, assuming that $B$ and $V$ cannot have both a zero row and a zero column and using Hammack's cancellation law, we give a necessary condition for the constructed graphs to be isomorphic. 

\begin{thm} \label{ref-bip-necc-cond} Let $G_{L}$ be a reflexive $(k,l)$-biregular bipartite graph with $k \neq l$. If the graphs $G_{L\underline{\otimes} H}$ and $G_{L\underline{\otimes} H^{\#}}$ isomorphic, then $B$ is PET matrix and the graphs $G_{A}$ and $G_{D}$ are isomorphic.
\end{thm}

\begin{proof} Let the graphs $G_{L\underline{\otimes} H}$ and $G_{L\underline{\otimes} H^{\#}}$ be isomorphic, then by the Lemma \ref{prop:eta} there exists an isomorphism between $G_{L\underline{\otimes} H}$ and $G_{L\underline{\otimes} H^{\#}}$ such that the corresponding permutation matrix $P$ satisfies $P^T(L\underline{\otimes}H)P= L\underline{\otimes} H^{\#}$, and $P$ is either of the form $\mathcal{I}(P_1,P_2)$ or $\mathcal{P}(P_1,P_2)$ for some permutation matrices $P_1$ and $P_2$. 
	
	Case 1: Suppose $ \mathcal{P} = \mathcal{I}(P_1,P_2)$. Then, we have $P_1^T(I_m\otimes A)P_1=I_m\otimes D$, $P_1^T(V\otimes B)P_2=V\otimes B^T$, and $P_2^T(I_n\otimes D)P_2=I_n\otimes A$. Using Hammack's cancellation law, the second equality implies that $B$ is PET. The other equalities imply that the graphs $G_{A}$ and $G_{D}$ are isomorphic.
	
	Case 2: Suppose $ \mathcal{P} = \mathcal{P}(P_1,P_2)$. Then, we have $P_1^T(I_m\otimes A)P_1=I_n\otimes A$, $P_2^T(V^T\otimes B^T)P_1=V\otimes B^T$ and $P_2^T(I_n\otimes D)P_2=I_m\otimes D$. Using Hammack's cancellation law, the second equality implies that $V$ is PET. But as $V$ has distinct row and column sums, this case cannot occur.
\end{proof}

For constructing cospectral nonisomorphic graphs, the condition that either $B$ is non-PET or $G_A$ and $G_D$ are nonisomorphic is sufficient. 
{The next example illustrates the construction.}
\begin{ex}
	Let $m=1$, $V=j_n^T$ where $j_n$ is the all-one vector of length $n>1$.  Then, $$L\underline{\otimes} H=\begin{bmatrix}A&B&B&\ldots&B\\B^T&D&0&\ldots&0\\B^T&0&D&\ldots&0\\\vdots&\vdots&\vdots&\ddots&\vdots\\B^T&0&0&\cdots&D\end{bmatrix}, \;L\underline{\otimes} H^{\#}=\begin{bmatrix}D&B^T&B^T&\ldots&B^T\\B&A&0&\ldots&0\\B&0&A&\ldots&0\\\vdots&\vdots&\vdots&\ddots&\vdots\\B&0&0&\cdots&A\end{bmatrix}. $$ 
	Let $A$ and $D$ be cospectral. Then, by Theorem \ref{gm-cospec}, the graphs $G_{L\underline{\otimes}H}$ and $G_{L\underline{\otimes}H^{\#}}$ are cospectral. Let $B=\begin{bmatrix}1&0\\1&0\end{bmatrix}$, and $A=D=\begin{bmatrix}0&1\\1&0\end{bmatrix}$. As $B$ is a non-PET matrix, by Theorem \ref{ref-bip-necc-cond}, the graphs $G_{L\underline{\otimes} H}$ and $G_{L\underline{\otimes} H^{\#}}$ nonisomorphic. See Figure \ref{i-b-ex} for the case $n=2$. 
	
	\begin{figure}[h!]
		\centering
		\begin{tabular}{ll}
			\includegraphics[scale=0.5]{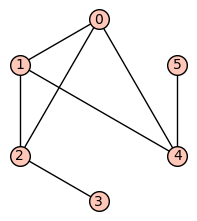}
			\hfill
			\includegraphics[scale=0.5]{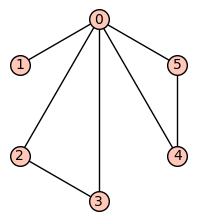}
		\end{tabular}
		\caption{Unfoldings of a graph with respect to a reflexive bipartite graph}
		\label{i-b-ex}
	\end{figure}

\end{ex}

\section{Construction II - Unfoldings involving a semi reflexive bipartite graph}\label{Semi_reflex_unfold} 

This section describes a cospectral construction for simple graphs by unfolding a semi-reflexive bipartite graph. Let $B$ and $U$ be matrices of size $p\times p$ and $m\times n$ respectively, and entries of $B$ and $U$ are either $0$ or $1$. Let $I_p$ be an identity matrix of order $p$, and $X$ be the adjacency matrix of a simple graph on $n$ vertices without any loops. Let $M=[M_{ij}]$ be a block matrix such that each $M_{ij}$ is a square matrix. The partial transpose of $M$, denoted by $M^{\tau}$, is defined as $M^{\tau} = [M_{ij}^T]$. 

Define $L=\begin{bmatrix}0&U\\U^T&X\end{bmatrix}$,  $H=\begin{bmatrix}0&B\\B^T&I_p\end{bmatrix}$, and  $H^{\tau}=\begin{bmatrix}0&B^T\\B&I_p\end{bmatrix}$. Consider the partitioned tensor products $$L\underline{\otimes} H=\begin{bmatrix}0&U\otimes  B\\U^T\otimes  B^T&X\otimes  I_p\end{bmatrix} ~\mbox{and}~ \;L\underline{\otimes}  H^{\tau}=\begin{bmatrix}0&U\otimes  B^T\\U^T\otimes  B&X\otimes  I_p\end{bmatrix}.$$ We call $G_{L\underline{\otimes} H}$ and $G_{L\underline{\otimes} H^{\tau}}$ the graphs obtained by unfolding the semi reflexive bipartite graph $G_H$ with respect to $G_L$. Since the graph $G_X$ does not have any loop, so $G_{L\underline{\otimes} H}$ and $G_{L\underline{\otimes} H^{\tau}}$ also do not have any loops. The vertex partition induced naturally by the partitioned tensor product on the graph $G_{L\underline{\otimes} H}$ is defined to be the canonical vertex partition of $G_{L\underline{\otimes} H}$.

\begin{thm}\label{thm-unf-semi-ref-cospec}
	The graphs $G_{L\underline{\otimes}H}$ and $G_{L\underline{\otimes}H^{\tau}}$ are cospectral for the adjacency matrix.  
\end{thm}
\begin{proof}
	Since $B$ is a square matrix, there exist two orthogonal matrices $Q_1$ and $Q_2$ such that $Q_1^TBQ_2=B^T$. Define $Q=\mathcal{I}(Q_1,Q_2)$. Then $Q$ is orthogonal and $Q^THQ=H^{\tau}$. 
	Let $R=\mathcal{I}(I_m,I_n)$, then $R$ satisfies $R^TLR=L$. Define $P=R\underline{\otimes} Q$. Then $P$ is an orthogonal matrix, and, by Proposition \ref{parti-diag},  $P^{T}(L\underline{\otimes} H)P=L\underline{\otimes} H^{\tau}$. Thus, $G_{L\underline{\otimes}H}$ and $G_{L\underline{\otimes}H^{\tau}}$ are cospectral. 
\end{proof}

Let $G_L\backslash G_X$ denote the bipartite graph obtained by removing all the edges in the induced subgraph $G_X$. Next, we prove a lemma, which helps us give an equivalent condition for the isomorphism of the  cospectral graphs constructed in the above theorem.

\begin{lem} \label{eta3-char}
	Let $G_X$ be a graph on $n$ vertices, and $G_L\backslash G_X$ be a $(k,l)$-biregular bipartite graph. Suppose one of the following holds: 
	\begin{enumerate}
		\item  $l\leq k$, and $G_X$ has no isolated vertices.
		\item $l>k$ and $G_X$ has maximum degree $l-k-1$. 
	\end{enumerate}
	If the graphs $G_{L\underline{\otimes}H}$ and $G_{L\underline{\otimes}H^{\tau}}$ are isomorphic, then the isomorphism respects the canonical vertex partitions of $G_{L\underline{\otimes}H}$ and $G_{L\underline{\otimes}H^{\tau}}$.
\end{lem}
\begin{proof}
	Let the graphs $\Gamma_1=G_{L\underline{\otimes}H}$ and $\Gamma_2=G_{L\underline{\otimes}H^{\tau}}$ be isomorphic and let $f$ be an isomorphism from $\Gamma_1$ to $\Gamma_2$. Let $V(\Gamma_i)=X_i\cup Y_i$ be the canonical vertex partitioning of the graphs $\Gamma_i$ for $i=1,2$. Let $b_i$ and $b_i'$ denote the $i^{th}$ row sum of the matrices $B$ and $B^T$, respectively. Let the degree sequence of the graph $G_X$ be  $a_1\geq a_2\geq\cdots\geq a_n \geq 0$.  
	
  Suppose $l\leq k$, and $a_n >0$. Then, let $x \in X_1$ be a vertex of minimum degreein $X_1$. Then, $f(x) \in X_2$. For, suppose that $f(x)\in Y_2$. Then, $d_{\Gamma_1}(x)=lb_i$ for some $1\leq i\leq p$ and $d_{\Gamma_2}(f(x))=kb_j+a_s$, $1\leq j\leq p$, $1 \leq s \leq n$. Since the isomorphism preserves the degrees, we have $lb_i=kb_j+a_s$. But this is a contradiction as $b_i \leq b_j$, $l\leq k$ and $a_s > 0$. Hence $f(x)\in X_2$. Let $x_1,\ldots, x_{rm}$ be the vertices in $X_1$ with the same minimum degree such that $d_{\Gamma_1}(x_{1+(s-1)m})=\ldots=d_{\Gamma_1}(x_{m+(s-1)m})=lb_{i_s}$ for $s\in \{1,2,\ldots, r\}$ where $b_{i_1}=\ldots=b_{i_r}$ for $1\leq i_1,\ldots, i_r\leq p$. Then, using the previous argument, for the vertices  $f(x_1),\ldots, f(x_{rm})$ we have $d_{\Gamma_2}(f(x_{1+(s-1)m}))=\ldots=d_{\Gamma_2}(f(x_{m+(s-1)m}))=lb'_{j_s}$ for $s\in \{1,2,\ldots, r\}$ where $b'_{j_1}=\ldots=b'_{j_r}$ for $1\leq j_1,\ldots, j_r\leq p$. Define $B'$ and $B''$ to be the matrices obtained by removing $i_s^{th}$ row and $j_s^{th}$ column respectively from $B$ for all $s\in \{1,2,\ldots, r\}$. Define $\Gamma_1'$ and $\Gamma_2'$ to be the induced graphs corresponding to the adjacency matrices $$\begin{bmatrix}0&U\otimes  B'\\U^T\otimes  B'^T&X\otimes  I\end{bmatrix}\; \text{and} \begin{bmatrix}0&U\otimes  B''^T\\U^T\otimes  B''&X\otimes  I\end{bmatrix},$$ respectively. Note that, the sizes of the matrices $B'$ and $B''$ are $(p-r)\times p$ and $p\times (p-r)$, respectively. Now, $\Gamma_1'$ and $\Gamma_2'$ are isomorphic as well, apply the same argument for $\Gamma_1'$ and $\Gamma_2'$ until both the graphs reduce to $G_{X\otimes I}$. Thus $f(X_1)=X_2$, and hence $f(Y_1)=Y_2$. 
	
 Suppose $l>k$ and $l-k>a_1$. Then, let $x \in X_1$ be a vertex of maximum degree in $X_1$. Then,  $f(x) \in X_2$. For,  suppose  that $f(x)\in Y_2$. Then, $d_{\Gamma_1}(x)=lb_i$ for some $1\leq i\leq p$ and $d_{\Gamma_2}(f(x))=kb_j+a_s$, $1\leq j\leq p$, $1 \leq s \leq n$. Since the isomorphism preserves the degrees, we have $lb_i=kb_j+a_s$. Since $x$ has maximum degree in $X_1$, we have $b_j\leq b_i$. This gives us $(l-k)b_i\leq a_s$, that is, $b_i \leq \frac{a_s}{l-k}$. Now, $a_s$ is the degree of a vertex in $G_X$, and we know that $\frac{a_s}{l-k}<1$. This implies that $b_i<1$. So $b_i=0$ and $B$ must be a zero matrix. We can choose $f(x)\in X_2$. Let $x_1,\ldots, x_{rm}$ be the vertices in $X_1$ with the same maximum degree:  $d_{\Gamma_1}(x_{1+(s-1)m})=\ldots=d_{\Gamma_1}(x_{m+(s-1)m})=lb_{i_s}$ for $s\in \{1,2,\ldots, r\}$ where $b_{i_1}=\ldots=b_{i_r}$ for $1\leq i_1,\ldots, i_r\leq p$. Then, using the previous argument, for the vertices $f(x_1),\ldots, f(x_{rm})$ we have $d_{\Gamma_2}(f(x_{1+(s-1)m}))=\ldots=d_{\Gamma_2}(f(x_{m+(s-1)m}))=lb'_{j_s}$ for $s\in \{1,2,\ldots, r\}$ where $b'_{j_1}=\ldots=b'_{j_r}$ for $1\leq j_1,\ldots, j_r\leq p$. Define $B'$ and $B''$ to be the matrices obtained by removing $i_s^{th}$ row and $j_s^{th}$ column respectively from $B$ for all $s\in \{1,2,\ldots, r\}$. Define $\Gamma_1'$ and $\Gamma_2'$ to be the induced graphs corresponding to the adjacency matrices $$\begin{bmatrix}0&U\otimes  B'\\U^T\otimes  B'^T&X\otimes  I\end{bmatrix}\; \text{and} \begin{bmatrix}0&U\otimes  B''^T\\U^T\otimes  B''&X\otimes  I\end{bmatrix},$$ respectively. Note that, the sizes of the matrices $B'$ and $B''$ are $(p-r)\times p$ and $p\times (p-r)$, respectively. Now, $\Gamma_1'$ and $\Gamma_2'$ are isomorphic as well, apply the same argument for $\Gamma_1'$ and $\Gamma_2'$ until both the graphs reduce to $G_{X\otimes I}$. Thus $f(X_1)=X_2$ and hence $f(Y_1)=Y_2$. 
\end{proof}

In the following theorem, we characterize the isomorphism of constructed cospectral graphs by assuming that the matrix $B$ cannot have both a zero row and a zero column and using Hammack's Cancellation Law.

\begin{thm}\label{thm-unf-ref-bip}
	Let $G_X$ be a graph on $n$ vertices, and $G_L\backslash G_X$ be a $(k,l)$-biregular bipartite graph. Suppose one of the following holds: 
	\begin{enumerate}
		\item  $l\leq k$, and $G_X$ has no isolated vertices.
		\item $l>k$ and $G_X$ has maximum degree $l-k-1$. \footnote{Since this construction requires that $G_X$ has at least one edge for it to be non-trivial, we can assume that $l-k-1>0$, that is, $l>k+1$.}
	\end{enumerate}
	Then, the graphs $G_{L\underline{\otimes} H}$ and $G_{L\underline{\otimes} H^{\tau}}$ are isomorphic if and only if $B$ is PET.
\end{thm}

\begin{proof}
	If $B$ is a PET matrix, then there exist two permutation matrices $Q_1$ and $Q_2$ such such that $Q_1^TBQ_2=B^T$. Define the permutation matrices $Q=\mathcal{I}(Q_1,Q_2)$ and $R=\mathcal{I}(I_m,I_n)$, then $Q^THQ=H^{\tau}$ and $R^TLR=L$. Now define $P=R\underline{\otimes}Q$, then $P^T(L\underline{\otimes}H)P=L\underline{\otimes}H^{\tau}$. Thus, the graphs $G_{L\underline{\otimes} H}$ and $G_{L\underline{\otimes} H^{\tau}}$ are isomorphic. 
	
	Conversely, suppose the graphs $G_{L\underline{\otimes} H}$ and $G_{L\underline{\otimes} H^{\tau}}$ are isomorphic. {Then, using Lemma \ref{eta3-char}}, there exists an isomorphism between $G_{L\underline{\otimes} H}$ and $G_{L\underline{\otimes} H^{\tau}}$ such that the corresponding permutation matrix that satisfies $P^T(L\underline{\otimes}H)P=L\underline{\otimes}H^{\tau}$ has the form $P=\mathcal{I}(P_1,P_2)$. Then, $P_1^T(U\otimes B)P_2=U\otimes B^T$ and $P_2^T(X\otimes I)P_2=X\otimes I$. Using Hammack's Cancellation Law $P_1^T(U\otimes B)P_2=U\otimes B^T$ implies that there exists two permutation matrices $R_1$ and $R_2$ such that $R_1^TBR_2=B^T$ and $B$ is PET. 
\end{proof}

We now demonstrate Theorem \ref{thm-unf-ref-bip} using examples. 

\begin{ex}
	
	Let $U=J_2$, $X=J_2-I_2$ and $B=\begin{bmatrix}1&1\\0&0\end{bmatrix}$. Then, $m=n=k=l=2$, and $B$ is a non-PET matrix. The constructed cospectral graphs $G_{L\underline{\otimes} H}$ and $G_{L\underline{\otimes} H^{\tau}}$ are nonisomorphic by the part (1) of Theorem \ref{thm-unf-ref-bip}. Figure \ref{ex-unf-semiref-bip1}. shows the unfoldings of a semi-reflexive bipartite graph corresponding to $B$ and given by the adjacency matrices 
	$$\begin{bmatrix}0&0&B&B\\0&0&B&B\\B^T&B^T&0&I\\B^T&B^T&I&0\end{bmatrix} \text{and}\begin{bmatrix}0&0&B^T&B^T\\0&0&B^T&B^T\\B&B&0&I\\B&B&I&0\end{bmatrix}. $$
	
	\begin{figure}[h!]
		\centering
		\includegraphics[scale=0.5]{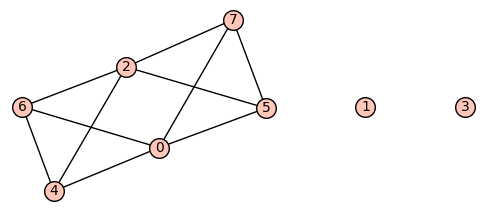}  \\
		\includegraphics[scale=0.5]{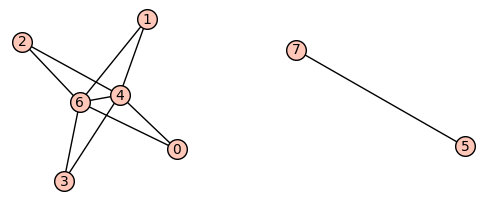}
		\caption{Cospectral nonisomorphic unfoldings of a semi-reflexive bipartite graph}
		\label{ex-unf-semiref-bip1}
	\end{figure}
	
\end{ex}

\begin{ex}
	Let $U=j_n^T$ be the all-one vector of length $n>1$ and $X$ be the adjacency matrix of a complete graph on $n$ vertices, that is, $X=J_n-I_n$. Then, $$L\underline{\otimes} H=\begin{bmatrix}0&B&B&\ldots&B\\B^T&0&I&\ldots&I\\B^T&I&0&\ldots&I\\\vdots&\vdots&\vdots&\ddots&\vdots\\B^T&I&I&\cdots&0\end{bmatrix}, \;L\underline{\otimes} H^{\tau}=\begin{bmatrix}0&B^T&B^T&\ldots&B^T\\B&0&I&\ldots&I\\B&I&0&\ldots&I\\\vdots&\vdots&\vdots&\ddots&\vdots\\B&I&I&\cdots&0\end{bmatrix}. $$ 
	Using Theorem \ref{thm-unf-semi-ref-cospec}, the graphs $G_{L\underline{\otimes}H}$ and $G_{L\underline{\otimes}H^{\tau}}$ are cospectral. This particular case is exactly the Construction III described by Kannan and Pragada in \cite{kannan-pragada}. 
	
	
	Let $U=j_3^T$, $X=\begin{bmatrix}0&1&0\\1&0&0\\0&0&0\end{bmatrix}$ and $B=\begin{bmatrix}1&1\\0&0\end{bmatrix}$. Then $k=m=1$, $l=n=3$, and $B$ is non-PET. The constructed cospectral graphs $G_{L\underline{\otimes} H}$ and $G_{L\underline{\otimes} H^{\tau}}$ are nonisomorphic, by  part (2) of Theorem \ref{thm-unf-ref-bip}. Figure \ref{ex-unf-semiref-bip} shows that the unfoldings of a semi-reflexive bipartite graph corresponding to $B$ and given by the adjacency matrices  
	$$\begin{bmatrix}0&B&B&B\\B^T&0&I&0\\B^T&I&0&0\\B^T&0&0&0\end{bmatrix} \text{and}\begin{bmatrix}0&B^T&B^T&B^T\\B&0&I&0\\B&I&0&0\\B&0&0&0\end{bmatrix}. $$
	\begin{figure}[H]
		\centering
		\includegraphics[scale=0.4]{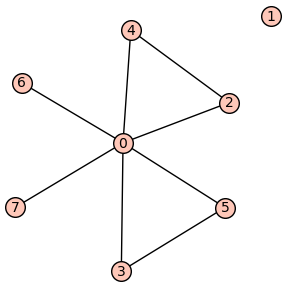}
		\\
		\includegraphics[scale=0.4]{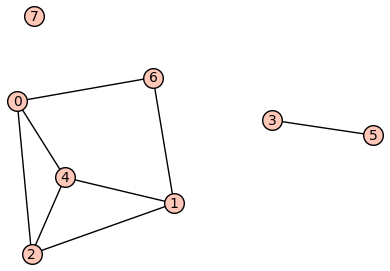} 
		\caption{Cospectral nonisomorphic unfoldings of a semi-reflexive bipartite graph}
		\label{ex-unf-semiref-bip}
	\end{figure}

	
	
\end{ex}

\section{Construction III - Unfoldings involving multipartite graphs}\label{Multipart_unfold}

In this section, we use the unfolding of multipartite graphs to construct cospectral graphs. Let us consider the matrices $$A=\begin{bmatrix} 0&B&B\\  B^T&0&B\\ B^T&B^T&0 \end{bmatrix}, A^{\tau}=\begin{bmatrix}0&B^T&B^T\\ B&0&B^T\\  B&B&0 \end{bmatrix}.$$ 

Define $L(p,q,r)=\begin{bmatrix} 0&J_{p,q}&J_{p,r}\\ J_{q,p}&0&J_{q,r}\\ J_{r,p}&J_{r,q}&0 \end{bmatrix},$ where $B$ is a $(0,1)$-matrix of order $n$, and $J_{m,n}$ is the all-one matrix of size $m\times n$. Now consider the partitioned tensor products
$$ L(p,q,r)\underline{\otimes}A=\begin{bmatrix} 0&J_{p,q}\otimes B&J_{p,r}\otimes B\\ J_{q,p}\otimes B^T&0&J_{q,r}\otimes B\\ J_{r,p}\otimes B^T&J_{r,q}\otimes B^T&0 \end{bmatrix}$$ and $$ L(p,q,r)\underline{\otimes}A^{\tau}=\begin{bmatrix} 0&J_{p,q}\otimes B^T&J_{p,r}\otimes B^T\\ J_{q,p}\otimes B&0&J_{q,r}\otimes B^T\\ J_{r,p}\otimes B&J_{r,q}\otimes B&0 \end{bmatrix}.$$

We say that the graphs $G_{L(p,q,r)\underline{\otimes}  A}$ and $G_{L(p,q,r)\underline{\otimes}  A^{\tau}}$ are unfoldings of the tripartite graph $G_A$ with respect to $G_{L(p,q,r)}$. A matrix $M$ is \emph{involutory} if $M^2 = I.$ By a result of Dutta and Adhikari \cite[Theorem 7]{dutta}, there exists a nonsingular matrix $Q$ such that $Q^{-1}AQ=A^{\tau}$ if $B$ is similar to its transpose by either an orthogonal matrix or a nonsingular involutory matrix. Next, we use this result involving partial transpose operation to give a new cospectral construction. 


\begin{lem}
	The graphs $G_{L(p,q,r)\underline{\otimes}  A}$ and $G_{L(p,q,r)\underline{\otimes}  A^{\tau}}$ are cospectral if $B$ is either orthogonally or (nonsingular) involutory similar to its transpose. 
\end{lem}
\begin{proof}
	Let $Q_0^TBQ_0=B^T$. If  $Q_0$ is an orthogonal matrix. Then,  $Q_0^TB^TQ_0=B$, and hence $Q_0^{-1}B^TQ_0=B$. If  $Q_0$ is involutory, then $Q_0^{-1}B^TQ_0=B$. Hence, we have $Q_0^{-1}BQ_0=B^T$ and $Q_0^{-1}B^TQ_0=B$. Consider the block diagonal nonsingular matrices $Q=\mathcal{I}(Q_0,Q_0,Q_0)$ and $R=\mathcal{I}(I_p,I_q,I_r)$. Then, $Q^{-1}AQ=A^{\tau}$ and  $R^{-1}L(p,q,r)R=L(p,q,r)$. Define $P=R\underline{\otimes}Q$, then $P$ is nonsingular and  $P^{-1}(L(p,q,r)\underline{\otimes}A)P=L(p,q,r)\underline{\otimes}A^{\tau}$. Thus, the corresponding graphs are cospectral.
\end{proof}

We now give some necessary and sufficient conditions for the cospectral graphs constructed to be isomorphic.  

\begin{thm}
	The graphs $G_{L(p,q,r)\underline{\otimes} A}$ and $G_{L(p,q,r)\underline{\otimes}  A^{\tau}}$ are isomorphic if either $p=r$ or $B$ is a PST matrix\footnote{Suppose $1\leq p\leq q\leq r$. Then, we can assume $p<r$ so that the constructed cospectral graphs are not isomorphic. This condition implies that $p\leq q <r$ or $p<q\leq r$ holds. }. Let $1\leq p< q< r$ and $p+q<r$. If the graphs $G_{L(p,q,r)\underline{\otimes} A}$ and $G_{L(p,q,r)\underline{\otimes}  A^{\tau}}$ are isomorphic, then $B$ is a PET matrix. 
\end{thm}
\begin{proof}
	Consider the following cases. 
	
	\textbf{Case 1:}  $B$ is a PST matrix. Then, there exists a permutation matrix $Q_0$ such that $Q_0^{-1}BQ_0=B^T$. Define $Q=\mathcal{I}(Q_0,Q_0,Q_0)$ and $R=\mathcal{I}(I_p,I_q,I_r)$. Then  $Q^{-1}AQ=A^{\tau}$ and $R^{-1}L(p,q,r)R=L(p,q,r)$. Define $P=R\underline{\otimes}Q$. 
	
	\textbf{Case 2:}  $p=r$. Define  $R=\mathcal{P}(I_p,I_q,I_r)$ and $Q=\mathcal{P}(I_n,I_n,I_n)$. Then  $R^{-1}L(p,q,r)R=L(p,q,r)$ and $Q^{-1}AQ=A^{\tau}$. Define $P=R\underline{\otimes}Q$. 
	
	In both the cases, the matrix $P$ is nonsingular, and satisfies $P^{-1}(L(p,q,r)\underline{\otimes}A)P=L(p,q,r)\underline{\otimes}A^{\tau}$. Thus, the graphs $G_{L(p,q,r)\underline{\otimes}  A}$ and $G_{L(p,q,r)\underline{\otimes}  A^{\tau}}$ are isomorphic. 
	
	Let the graphs $\Gamma_1 =G_{L(p,q,r)\underline{\otimes} A}$ and $\Gamma_2=G_{L(p,q,r)\underline{\otimes}  A^{\tau}}$ be isomorphic. Let $V(\Gamma_i)=X_i\cup Y_i\cup Z_i$ be the canonical vertex partition of the graphs $\Gamma_i$ for $i=1,2$. Let $f$ be an isomorphism from $\Gamma_1$ to $\Gamma_2$, and let $b_i$ and $b_i'$ denote the $i^{th}$ row sum of the matrices $B$ and $B^T$, respectively. 
	
	Let $x_1\in X_1$ be the vertex of maximum degree in $X_1$. Suppose that $f(x_1)\in Z_2$. Then $d_{\Gamma_1}(x_1)=(q+r)b_i$ for some $1\leq i \leq n$, and $d_{\Gamma_2}(f(x_1))=(p+q)b_j$ for some $1\leq j \leq n$. Since the isomorphism preserves the degree, we have $(q+r)b_i=(p+q)b_j$. Since $x_1$ has maximum degree in $X_1$, $b_i\geq b_j$ for any $1\leq j\leq n$, and hence $(p+q)b_j\geq (q+r)b_j$. If $b_j \neq 0$, then $p\geq r$, which contradicts the initial assumption that $p<r$. Hence, if $x_1\in X_1$, then $f(x_1)\notin Z_2$. If $b_j=0$, then since $(q+r)b_i=(p+q)b_j$, $b_i=0$. But $x_1$ is a vertex of maximum degree $(q+r)b_i$ in the set $X_1$, so $B=0$. So we could choose $f(x_1)\notin Z_2$. In any case, $f(x_1)\notin Z_2$. By the repeated removal of maximum degree vertices argument as in Lemma \ref{eta3-char}, we can conclude that $f(X_1)\cap Z_2 =\emptyset$. The supposition \lq $x_1\in X_2$ and $f^{-1}(x_1)\in Z_1$ for a vertex $x_1$ of maximum degree in the set $X_2$ \rq \;  contradicts the assumption $p<r$, and so we get $f^{-1}(X_2)\cap Z_1=\emptyset$. 
	
	Similarly, the suppositions \lq$y_1\in Y_1$ and $f(y_1)\in Z_2$ for a vertex $y_1$ of maximum degree in the set $Y_1$\rq \; and \lq$y_1\in Y_2$ and $f^{-1}(y_1)\in Z_1$ for a vertex $y_1$ of maximum degree in the set $Y_2$\rq \; both contradict the assumption $p+q<r$, and so we get $f(Y_1)\cap Z_2=\emptyset$ and $f^{-1}(Y_2)\cap Z_1=\emptyset$ respectively. Hence, $f(Z_1)=Z_2$ and $f(X_1\cup Y_1)=X_2\cup Y_2$. This shows the bipartite graphs induced by the sets $X_1\cup Y_1$ and $X_2\cup Y_2$ are isomorphic. Since $p\neq q$, using Corollary 4.7 \cite{kph-unf}, or Theorem 3.1 \cite{ji-gong-wang}, we can conclude that $B$ is a PET matrix. 
\end{proof}

\begin{ex}
	Let $B=\begin{bmatrix}1&1\\0&0\end{bmatrix}$ be a non-PST matrix. The corresponding graphs $G_A$ and $G^{A^\tau}$ are given in Figure 4. 
	
	\begin{figure}[h!]
		\centering
		\includegraphics[scale=0.4]{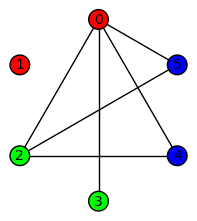}
		\includegraphics[scale=0.4]{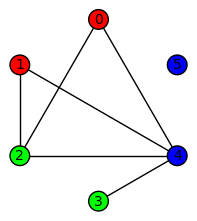}
		\caption{Tripartite graphs $G_A$ and $G_{A^{\tau}}$}
		\label{gh-ght}
	\end{figure}

	Note that  $Q=\frac{1}{\sqrt{2}}\begin{bmatrix}1&1\\1&-1\end{bmatrix}$ is orthogonal as well as involutory matrix satisfying $Q^{-1}BQ=B^T$. In Table \ref{lgh-lght}, the examples of cospectral non-isomorphic graphs $G_{L(p,q,r)\underline{\otimes} A}$ and $G_{L(p,q,r)\underline{\otimes} A^{\tau}}$ are generated for each tuple $(p,q,r)$ where $p\neq r$. Each tuple $(p,q,r)$ corresponds to a different way of unfolding the given tripartite graph $G_A$. 
	\begin{table}[h!]
		\centering
		\begin{tabular}{ | c| c| c | } 
			\hline
			(p,q,r)& $G_{L(p,q,r)\underline{\otimes} A}$ & $G_{L(p,q,r)\underline{\otimes} A^{\tau}}$  \\ 
			\hline
			(1,1,2) & \includegraphics[scale=0.4]{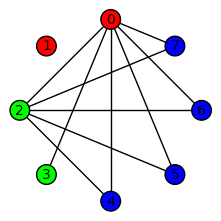} & \includegraphics[scale=0.4]{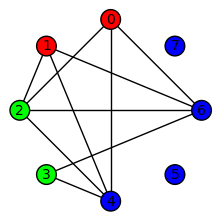} \\ 
			\hline
			(1,1,3) & \includegraphics[scale=0.4]{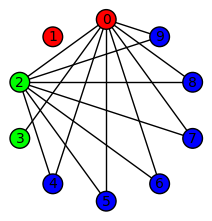} & \includegraphics[scale=0.4]{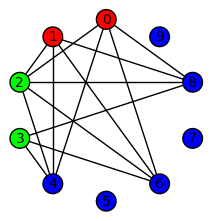} \\ 
			\hline
			(1,2,2) & \includegraphics[scale=0.4]{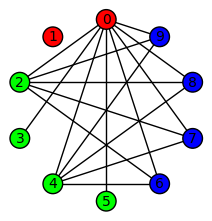} & \includegraphics[scale=0.4]{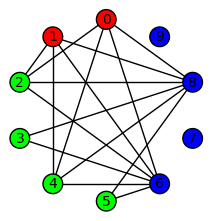} \\ 
			\hline
			(1,2,3) & \includegraphics[scale=0.4]{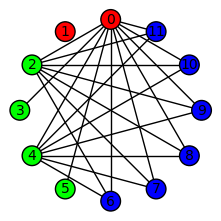} & \includegraphics[scale=0.4]{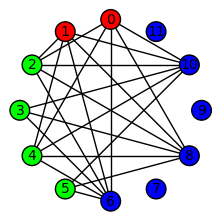} \\ 
			\hline
			(1,3,3) & \includegraphics[scale=0.4]{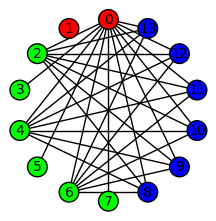}& \includegraphics[scale=0.4]{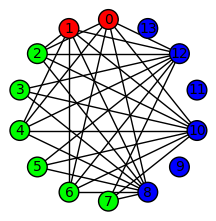}\\ 
			\hline
		\end{tabular}
		\caption{Unfoldings of tripartite graphs $G_{L(p,q,r)\underline{\otimes} A}$ and $G_{L(p,q,r)\otimes A^{\tau}}$}
		\label{lgh-lght}
	\end{table}
\end{ex}
\section*{Acknowledgments}
The authors thank the referees for the comments and suggestions.
\bibliographystyle{plain}
\nocite{hitesh}
\bibliography{references}

\begin{thebibliography}{10}

\bibitem{arvind2023hierarchy}
V.~Arvind, Frank Fuhlbr\"{u}ck, Johannes K\"{o}bler, and Oleg Verbitsky.
\newblock {On a Hierarchy of Spectral Invariants for Graphs}.
\newblock In Olaf Beyersdorff, Mamadou~Moustapha Kant\'{e}, Orna Kupferman, and
  Daniel Lokshtanov, editors, {\em 41st International Symposium on Theoretical
  Aspects of Computer Science (STACS 2024)}, volume 289 of {\em Leibniz
  International Proceedings in Informatics (LIPIcs)}, pages 6:1--6:18,
  Dagstuhl, Germany, 2024. Schloss Dagstuhl -- Leibniz-Zentrum f{\"u}r
  Informatik.

\bibitem{but-lama}
Steve Butler.
\newblock A note about cospectral graphs for the adjacency and normalized
  {L}aplacian matrices.
\newblock {\em Linear Multilinear Algebra}, 58(3-4):387--390, 2010.

\bibitem{dutta}
Supriyo Dutta and Bibhas Adhikari.
\newblock Construction of cospectral graphs.
\newblock {\em Journal of Algebraic Combinatorics}, 52(2):215--235, September
  2020.

\bibitem{godsil-mckay-1976}
C.~Godsil and B.~Mckay.
\newblock Products of graphs and their spectra.
\newblock In Louis R.~A. Casse and Walter~D. Wallis, editors, {\em
  Combinatorial {Mathematics} {IV}}, Lecture {Notes} in {Mathematics}, pages
  61--72, Berlin, Heidelberg, 1976. Springer.

\bibitem{godsil-mckay-1982}
C.~D. Godsil and B.~D. McKay.
\newblock Constructing cospectral graphs.
\newblock {\em Aequationes Math.}, 25(2-3):257--268, 1982.

\bibitem{haemers_are_2016}
W.~H. Haemers.
\newblock Are almost all graphs determined by their spectrum?
\newblock {\em Notices of the South African Mathematical Society}, 47:42--45,
  2016.

\bibitem{hammack}
Richard~H. Hammack.
\newblock Proof of a conjecture concerning the direct product of bipartite
  graphs.
\newblock {\em European Journal of Combinatorics}, 30(5):1114--1118, July 2009.

\bibitem{ji-gong-wang}
Yizhe Ji, Shicai Gong, and Wei Wang.
\newblock Constructing cospectral bipartite graphs.
\newblock {\em Discrete Mathematics}, 343(10):112020, October 2020.

\bibitem{kannan-pragada}
M.~Rajesh Kannan and Shivaramakrishna Pragada.
\newblock On the construction of cospectral graphs for the adjacency and the
  normalized {L}aplacian matrices.
\newblock {\em Linear Multilinear Algebra}, 70(15):3009--3030, 2022.

\bibitem{kph-unf}
M.~Rajesh Kannan, Shivaramakrishna Pragada, and Hitesh Wankhede.
\newblock On the construction of cospectral nonisomorphic bipartite graphs.
\newblock {\em Discrete Math.}, 345(8):Paper No. 112916, 8, 2022.

\bibitem{qiu2020theorem}
Lihong Qiu, Yizhe Ji, and Wei Wang.
\newblock On a theorem of {G}odsil and {M}c{K}ay concerning the construction of
  cospectral graphs.
\newblock {\em Linear Algebra Appl.}, 603:265--274, 2020.

\bibitem{vandam-haemers}
Edwin~R. van Dam and Willem~H. Haemers.
\newblock Which graphs are determined by their spectrum?
\newblock volume 373, pages 241--272. 2003.
\newblock Special issue on the Combinatorial Matrix Theory Conference (Pohang,
  2002).

\bibitem{devel-vandam-haemers}
Edwin~R. van Dam and Willem~H. Haemers.
\newblock Developments on spectral characterizations of graphs.
\newblock {\em Discrete Math.}, 309(3):576--586, 2009.

\bibitem{wang2019cospectral}
Wei Wang, Lihong Qiu, and Yulin Hu.
\newblock Cospectral graphs, {GM}-switching and regular rational orthogonal
  matrices of level {$p$}.
\newblock {\em Linear Algebra Appl.}, 563:154--177, 2019.

\bibitem{hitesh}
Hitesh Wankhede.
\newblock {\em Constructing cospectral graphs using partitioned tensor
  product}.
\newblock MS Thesis. IISER Pune, 2021.

\end{thebibliography}
\end{document}